\documentclass[leqno,11pt,a4paper]{article}
\usepackage[left=3.0cm, top=2.2cm,bottom=2.5cm]{geometry}
\usepackage{amsmath,amssymb,amsthm,mathrsfs,calc,graphics,color}
\usepackage{graphicx}
\usepackage[british]{babel}

\newtheorem {theorem}{Theorem}

\newtheorem {lemma}[theorem]{Lemma}

{\theoremstyle{definition}

}
{\theoremstyle{theorem}

}

\newcommand{\PHT}{\operatorname{PHT}}

\def\ba{\begin{array}}
\def\ea{\end{array}}
\def\bea{\begin{eqnarray} \label}
\def\eea{\end{eqnarray}}
\def\be{\begin{equation} \label}
\def\ee{\end{equation}}
\def\bit{\begin{itemize}}
\def\eit{\end{itemize}}
\def\ben{\begin{enumerate}}
\def\een{\end{enumerate}}

\def\RR{\mathbb{R}}

\def\r{\varrho}

\def\ph{\varphi}

\def\L{\Lambda}

\def\bE{\mathbf{E}}

\def\bP{\mathbf{P}}
\def\bQ{\mathbf{Q}}
\def\bR{\mathbf{R}}

\def\cH{\mathcal{H}}

\def\cP{\mathscr{P}}

\def\cS{\mathcal{S}}

\def\cMP{\mathcal{MP}}
\def\MP{\textup{MP}}
\def\F{\textup{F}}
\def\M{\textup{M}}

\def\dint{\textup{d}}
\def\lk{\langle}
\def\rk{\rangle}

\begin{document}

\title{\bfseries Birth-time distributions of weighted polytopes in STIT tessellations}

\author{Nguyen Ngoc Linh\footnotemark[1]\, and Christoph Th\"ale\footnotemark[2]}

\date{}
\renewcommand{\thefootnote}{\fnsymbol{footnote}}
\footnotetext[1]{Institut f\"ur Stochastik der Friedrich-Schiller-Universit\"at Jena, Ernst-Abbe-Platz 2, 07743 Jena, Germany. E-mail: linh.nguyen@uni-jena.de}

\footnotetext[2]{
Institut f\"ur Mathematik der Universit\"at Osnabr\"uck, Albrechtstra\ss e 28a, 49076 Osnabr\"uck, Germany. E-mail: christoph.thaele@uni-osnabrueck.de}

\maketitle

\begin{abstract}
The lower-dimensional maximal polytopes associated with an iteration stable (STIT) tessellation in $\RR^d$ are considered. They arise in the spatio-temporal construction process of such a tessellation as intersections of $(d-1)$-dimensional maximal polytopes. A precise description of the joint distribution of their birth-times is obtained. This in turn is used to determine the probabilities that the typical or the length-weighted typical maximal segment of the tessellation contains a fixed number of internal vertices.
\bigskip
\\
{\bf Keywords}. {Birth-time, internal vertex, iteration/nesting, marked point process, maximal segment, Poisson process, random polytope, STIT tessellation, stochastic geometry, stochastic stability, weighted polytope.}\\
{\bf MSC}. Primary  60D05; Secondary 60G55, 60J75.
\end{abstract}

\section{Introduction and results}

Random tessellation theory is an active field of current mathematical research. Besides theoretical developments, there are numerous applications of random tessellations for example in the study of random structures in biology, geology and other sciences; cf.~\cite{Beil0,Lau,OkaBooSugChi00,Red,SKM}. Apart from the classical models, such as the well known Poisson hyperplane or Poisson-Voronoi tessellations for which we refer to \cite{Calka,OkaBooSugChi00,SW,SKM}, the class of iteration stable (or STIT) tessellations has attracted considerable interest in recent times; see \cite{MNW,MNW11,NW05,STBern,STAOP,T10,TWN} and the references cited therein. In particular and as discussed in \cite{NMOW}, the STIT tessellations may serve as a reference model for hierarchical spatial cell-splitting and crack formation processes in natural sciences and technology, for example to describe geological or material phenomena or aging processes of surfaces.

Intuitively and within a compact convex polytope $W\subset\RR^d$ (we assume that $d\geq 2$ in this paper) with positive volume, their construction can be described as follows. At first, the window $W$ is equipped with a random lifetime. When the lifetime of $W$ runs out, we choose a hyperplane $H$, which divides $W$ into two non-empty sub-polytopes $W^+$ and $W^-$. Now, the construction continues independently and recursively in $W^+$ and $W^-$ until some fixed time threshold $t>0$ is reached; see Figure \ref{fig1}. The outcome $Y_W(t)$ of this algorithm is a random subdivision (tessellation) of $W$ into polytopes (called cells in the sequel).
\begin{figure}[t]
\begin{center}
 \includegraphics[width=11cm]{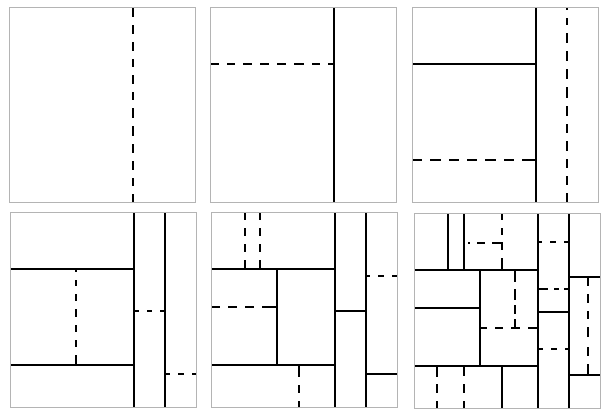}
 \caption{States of the random cell division process at different time instants; the respective new segments are dashed. Here, $W$ is a square.}\label{fig1}
\end{center}
\end{figure}
To be more precise, we have to specify the lifetime distribution of the cells and the law of the cell-separating hyperplanes. For this purpose, let us write $\cH$ for the space of affine hyperplanes and $\cH_0$ for the subspace of linear hyperplanes in $\RR^d$. Furthermore, let $\L$ be a measure on $\cH$, which admits the decomposition
\begin{equation}\label{eq:decomplambda}
\int_{\cH}f(H)\,\L(\dint H) = \int_{\cH_0}\int_{H_0^\perp}f(H_0+x)\,\ell_{H_0^\perp}(\dint x)\,\bQ(\dint H_0),
\end{equation}
where $f$ is a non-negative measurable function on $\cH$, $\ell_{H_0^\perp}$ is the Lebesgue measure on $H_0^\perp$ and $\bQ$ is a probability measure on $\cH_0$. We require that $\bQ$ is non-degenerate in the sense that ${\rm span}\{{n}(H):H\in{\rm supp}\,\bQ\}=\RR^d$, where ${n}(H)$ is the unit normal of $H$ lying (for definiteness) in the upper unit half-sphere. For a polytope $c$ write $\lk c\rk=\{H\in\cH:c\cap H\neq\emptyset\}$ for the collection of hyperplanes hitting $c$. Now, the lifetime of a cell $c$ in the above construction is chosen to be exponentially distributed with mean $\L(\lk c\rk)^{-1}$ (if $\bQ$ is the uniform distribution on $\cH_0$, $\L(\lk c\rk)$ is just a dimension dependent multiple of the mean width of $c$). Moreover, we choose the hyperplane splitting $c$ according to the (conditional) law $\Lambda(\,\cdot\,|\lk c\rk)$. It is exactly this choice, which makes mathematical analysis of the so-constructed tessellations possible and completes the description of the above algorithm; see \cite{NW05} for more details. In particular, we notice that the exponential lifetimes of the cells ensure that the construction enjoys the Markov property in the continuous time parameter $t$.

Besides looking at the local tessellation $Y_W(t)$ within $W$, it is convenient to extend $Y_W(t)$ to a whole space random tessellation $Y(t)$ in such a way that for any $W$ as above, $Y(t)$ restricted to $W$ has the same distribution as the previously constructed $Y_W(t)$ (this is possible by consistency according to the main result in \cite{NW05}). We call $Y(t)$ a \textbf{STIT tessellation} of $\RR^d$ since $Y(t)$ enjoys a stochastic \underline{st}ability under \underline{it}erations as explained later (that $Y(t)$ is in fact a tessellation is ensured by the non-degeneracy assumption on $\bQ$). It is important that the form \eqref{eq:decomplambda} of $\L$ implies that the distribution of $Y(t)$ is invariant under spatial translations, i.e., the shifted tessellation $Y(t)+x$ has the same distribution as $Y(t)$ for any $x\in\RR^d$.

With $Y(t)$, a number of geometric objects are associated. To introduce them, we write $\cMP_{d-1}^{(t)}$ for the family of cell-splitting hyperplane pieces that are introduced during the recursion steps in the above algorithm until time $t$ (these are the dashed segments in Figure \ref{fig1}). More generally, for $k=0,\ldots,d-2$ we denote by $\cMP_k^{(t)}$ the family of $k$-dimensional faces of members of $\cMP_{d-1}^{(t)}$. For $k=0,\ldots,d-1$ we call $\cMP_k^{(t)}$ the class of \textbf{$k$-dimensional maximal polytopes} of $Y(t)$. They are the natural building blocks of $Y(t)$ and its lower-dimensional face-skeletons. We also consider $k$-dimensional weighted maximal polytopes, where the intrinsic volumes $V_j$, $0\le j \le k$, constitute the weights. To define them, fix $k\in\{0,\ldots,d-1\}$ and $j\in\{0,\ldots,k\}$, let $m(p)$ be the circumcenter of a polytope $p$, write $\cP_k^o$ for the (measurable) space of $k$-polytopes with circumcenter at the origin $o$ and let $B_r^d$ be the centered ball in $\RR^d$ with radius $r>0$. Now, we introduce a probability measure $\bP_{k,j}^{(t)}$ on $\cP_k^o$ as follows:
\begin{equation}\label{eq:defweightedpolytope}
 \bP_{k,j}^{(t)}(A) :=\lim\limits_{r\rightarrow\infty}\frac{\bE\sum\limits_{p\,\in\cMP_k^{(t)}}{\bf 1}\{p\subset B_r^d\}\,{\bf 1}\{p-m(p)\in A\}\,V_j(p)}{\bE\sum\limits_{p\,\in\cMP_k^{(t)}}{\bf 1}\{p\subset B_r^d\}\,V_j(p)},
\end{equation}
where $A$ is a Borel subset of $\cP_k^o$ (following the proof of Equations (10) and (11) in \cite{Schneider09}, it can be shown that the limit is well-defined; another argument can be given by using \eqref{eq:ProofGeneralBirthTime1} below). A random polytope with distribution $\bP_{k,j}^{(t)}$ is called a \textbf{$V_j$-weighted typical $k$-dimensional maximal polytope} of $Y(t)$ and will henceforth be denoted by $\MP_{k,j}^{(t)}$. If $j=0$, this is the typical $k$-dimensional maximal polytope and  for $j=k$ we obtain the volume-weighted typical $k$-dimensional maximal polytope of the STIT tessellation $Y(t)$, which are two classical objects considered in stochastic geometry; cf.~\cite{Calka,Schneider09,SW}. For example, $\MP_{1,0}^{(t)}$ is the \textbf{typical maximal segment}, whereas $\MP_{1,1}^{(t)}$ is the \textbf{length-weighted typical maximal-segment}. However, our approach here is more general and interpolates between the typical ($j=0$) and the volume-weighted typical ($j=k$) maximal polytope of dimension $k$.

Any $k$-dimensional maximal polytope of $Y(t)$ is by definition the intersection of $d-k$ maximal polytopes of dimension $d-1$. In view of the spatio-temporal construction described above, each of these polytopes has a well-defined random birth-time. We denote these random variables by $\beta_1,\ldots,\beta_{d-k}$ and order them in such a way that $0<\beta_1<\ldots<\beta_{d-k}<t$ holds almost surely. Our first result describes the joint distribution of these birth-times; in the special cases $d=2$ or $d=3$, $k=1$ and $j=0$ or $j=1$, the formula is known from \cite{MNW11,T10,TWN}.

\begin{theorem}\label{thm:GeneralBirthTime}
Given $d\geq 1$, $k\in\{0,\ldots,d-1\}$ and $j\in\{0,\ldots,k\}$. The joint distribution of the birth-times $\beta_1,\ldots,\beta_{d-k}$ of the $V_j$-weighted typical $k$-dimensional maximal polytope of the STIT tessellation $Y(t)$ has density
$$(s_1,\ldots,s_{d-k})\mapsto (d-j)(d-k-1)!{s_{d-k}^{k-j}\over t^{d-j}}\,{\bf 1}\{0<s_1<\ldots<s_{d-k}<t\}$$
with respect to the Lebesgue measure on the $(d-k)$-dimensional simplex $\{0<s_1<\ldots<s_{d-k}<t\}$, which is independent of the hyperplane measure $\Lambda$. In particular, if $j=k$ we obtain the uniform distribution.
\end{theorem}

We turn now to an application of Theorem \ref{thm:GeneralBirthTime}, where we consider the typical and the length-weighted typical maximal segment $\MP_{1,0}^{(t)}$ and $\MP_{1,1}^{(t)}$ of $Y(t)$, respectively. These segments may have internal vertices, which arise at the time of birth of the segment (when $d\ge 3$) and thereafter subject to further subdivision of adjacent cells; see Figure \ref{fig1} for an illustration in the planar case. With the help of Theorem \ref{thm:GeneralBirthTime} we can determine the probabilities $\mathsf{p}_{1,0}(n)$ and $\mathsf{p}_{1,1}(n)$ that the typical or the length-weighted typical maximal segment of $Y(t)$ contains exactly $n\in\{0,1,2,\ldots\}$ internal vertices (we suppress the dependency on $t$ in the notation of these probabilities since they are independent of the time parameter $t$).

\begin{theorem}\label{thm:GeneralInternalVertices} The probabilities $\mathsf{p}_{1,0}(n)$ and $\mathsf{p}_{1,1}(n)$ are given by
$$\mathsf{p}_{1,0}(n)=d(d-2)!\int_0^t\int_0^{s_{d-1}}\!\!\!\!\!\!\!\!\ldots\int_0^{s_2}{s_{d-1}^2\over t^d}{(d\cdot t-2s_{d-1}-s_{d-2}-\ldots-s_1)^n\over(d\cdot t-s_{d-1}-s_{d-2}-\ldots-s_1)^{n+1}}\,\dint s_1\ldots\dint s_{d-1}$$ and
\begin{equation*}
\begin{split}
\mathsf{p}_{1,1}(n)=(n+1)(d-1)!\int_0^t\int_0^{s_{d-1}}\!\!\!\!\!\!\!\!\ldots\int_0^{s_2}{s_{d-1}^2\over t^{d-1}}&{(d\cdot t-2s_{d-1}-s_{d-2}-\ldots-s_1)^n\over(d\cdot t-s_{d-1}-s_{d-2}-\ldots-s_1)^{n+2}}\\ &\dint s_1\ldots\dint s_{d-1}.
\end{split}
\end{equation*}
Moreover, $\mathsf{p}_{1,0}(n)$ and $\mathsf{p}_{1,1}(n)$ are
independent of $t$ and $\L$. In the mean, the typical maximal segment has ${1\over 2}(d^2-d+2)/(d-1)$ internal vertices in dimension $d\geq 2$, whereas the length-weighted typical maximal segment in space dimension $d\geq 3$ has $(d^2-2d+4)/(d-2)$ (the mean is infinite if $d=2$).
\end{theorem}

In the planar case $d=2$, $\mathsf{p}_{1,0}(n)$ is known from \cite{MNW11,T10}, whereas for $d=3$ the formula for $\mathsf{p}_{1,0}(n)$ has been established in \cite{TWN} by different methods. Our approach in the present paper is more general and allows to deduce the corresponding formula for the length-weighted maximal segment as well as to deal with arbitrary space dimensions. To provide a concrete example, take $d=3$ and consider the length-weighted typical maximal segment. Here, we have
\begin{eqnarray*}
\mathsf{p}_{1,1}(0) &=& 5+18\ln 2-{63\over 4}\ln 3\approx 0.173506\,,\\
\mathsf{p}_{1,1}(1) &=& 28+90\ln 2-{657\over 8}\ln 3\approx 0.159712,\qquad {\rm etc.}
\end{eqnarray*}
The mean number of internal vertices equals $7$ in this case.

\bigskip

The rest of this note is structured as follows. In Section \ref{sec:preliminaries} we recall some important facts about STIT and also about Poisson hyperplane tessellations, which prepare for the proof of our results. These are presented in the two final sections \ref{sec:proofs} and \ref{sec:proofs2}.

\section{Some preliminaries}\label{sec:preliminaries}

\emph{Iteration of tessellations}.
Below we will exploit the fact that the tessellations $Y(t)$ are stable under iterations. To explain what this means, let $0<s<t$ and define the \textbf{iteration} $Y(s)\boxplus Y(t)$ of $Y(s)$ and $Y(t)$ as the tessellation that arises by locally superimposing within the cells of $Y(s)$ independent copies of $Y(t)$. Formally, let $\{Y_c(t):c\in Y(s)\}$ be a family of i.i.d.\ copies of $Y(t)$, which is indexed by the cells of $Y(s)$, and which is also independent of $Y(s)$. Then, $$Y(s)\boxplus Y(t):=\big\{c\cap \tilde{c}:c\in Y(s),\ \tilde{c}\in Y_c(t),\,\textup{int}\,c\cap\textup{int}\,\tilde{c}\neq\emptyset\big\}.$$ The STIT tessellations are stable under iteration in that the distributional equality
\begin{equation}\label{eq:iterationstability}
 Y(s)\boxplus Y(t) \overset{d}{=} Y(s+t)
\end{equation}
holds for any $s,t>0$. In other words, the results are the same in distribution when we either run the above cell-division algorithm from time $s$ to time $s+t$ or perform at time $s$ an iteration of $Y(s)$ and $Y(t)$; cf.~\cite{MNW,NW05}. This will play an important r\^ole in the proofs below.

\bigskip

\emph{STIT scaling}.
We collect here two implications of the scaling property of a STIT tessellation $Y(t)$. Globally, it says that the dilation $tY(t)$ of $Y(t)$ by factor $t$ has the same distribution as $Y(1)$, the STIT tessellation with time parameter $1$, i.e.,
\begin{equation}\label{eq:STITglobalScaling}
 tY(t)\overset{d}{=}Y(1).
\end{equation}
For a polytope $W$ we also have the local scaling $tY_{W}(t)\overset{d}{=}Y_{tW}(1)$; see \cite{NW05} for example.

Let us denote by $\r_{k,j}^{(t)}$ the density of the $j$-th intrinsic volume of $\cMP_k^{(t)}$, that is,
\begin{equation}\label{eq:rkj}
\r_{k,j}^{(t)}:=\lim_{r\rightarrow\infty}{1\over r^d\kappa_d}\bE\sum_{p\in\cMP_k^{(t)}}{\bf 1}\{\,p\subset B_r^d\,\}\, V_j(p),
\end{equation}
where $\kappa_d$ is the volume of $B_1^d$ (one can in fact show that this limit is well-defined, see \cite{STBern}). Using \eqref{eq:STITglobalScaling}, the definitions \eqref{eq:defweightedpolytope} of $\MP_{k,j}^{(t)}$ and \eqref{eq:rkj} of $\r_{k,j}^{(t)}$ as well as the homogeneity of  the intrinsic volumes, one easily shows the following two facts.
\begin{lemma}\label{lem:Scaling}
For $t>0$, $k\in\{0,\ldots,d-1\}$ and $i,j\in\{0,\ldots,k\}$ it holds that
\begin{itemize}
 \item[a)] $\r_{k,j}^{(t)}=t^{d-j}\r_{k,j}^{(1)}$,
 \item[b)] $\bE V_i(\MP_{k,j}^{(t)})=t^{-i}\,\bE V_i(\MP_{k,j}^{(1)})$.
\end{itemize}
\end{lemma}
Besides these two scaling relations, also the exact values of $\r_{k,j}^{(t)}$ and $\bE V_i(\MP_{k,j}^{(t)})$ are known from \cite{STBern} or can be determined with the help of Lemma \ref{lem:Vjbirthtime} below, but they are not important for our purposes.

\bigskip

\emph{STIT intersections.} Let $E$ be a $k$-dimensional affine subspace of $\RR^d$. Then $Y(t)\cap E$ is a STIT tessellation within $E$, i.e., the sectional tessellation $Y(t)\cap E$ is also stable under iterations. It has the property that the $(k-1)$-volume density of its cell boundaries is proportional to $t$ with a proportionality constant depending on $d$, $k$ and the hyperplane measure $\L$ (this follows from standard intersection formulae for surface processes \cite[Theorem 4.5.3]{SW}). In particular, the intersection of the STIT tessellation $Y(t)$ with a line $L={\rm span}\,u$, where $u\in{\cal S}_+^{d-1}$ (upper unit half-sphere), is a Poisson point process with intensity $t\L(\lk u\rk)$ (here $u$ has to be interpreted as the line segment connecting the origin with $u$); cf.~\cite{NW05}.

\bigskip

\emph{Poisson hyperplanes}.
Let us denote by $\PHT(t)$ a \textbf{Poisson hyperplane tessellation} in $\RR^d$ with intensity measure $t\L$; cf.~\cite{Calka,SW} for definition. Now, similarly as for the STIT tessellations, for $k\in\{0,\ldots,d-1\}$ and $j\in\{0,\ldots,k\}$ we denote by $\F_{k,j}^{\PHT(t)}$ the \textbf{$V_j$-weighted typical $k$-face} of $\PHT(t)$. Formally, its distribution is defined by replacing the class $\cMP_k^{(t)}$ of $k$-dimensional maximal polytopes of $Y(t)$ in \eqref{eq:defweightedpolytope} by the class of $k$-faces of the Poisson hyperplane tessellation $\PHT(t)$.

We make use of the scaling property of Poisson hyperplane tessellation, which is similar to the scaling property (\ref{eq:STITglobalScaling}) of a STIT tessellation. Formally, it says that the dilated Poisson hyperplane tessellation $t\PHT(t)$ has the same distribution as $\PHT(1)$, i.e.,
\begin{equation}\label{eq:PHTScaling}
 t\PHT(t)\overset{d}{=}\PHT(1).
\end{equation}
This follows directly from the uniqueness theorem for Poisson processes \cite[Theorem 3.2.1]{SW} and the form of the intensity measure $\L$; recall \eqref{eq:decomplambda}.

\bigskip

\emph{A distributional identity}.
In our arguments below, we need an identity describing the distribution of $\MP_{k,j}^{(t)}$ in terms of weighted faces in Poisson hyperplane tessellations. In \cite[Theorem 3]{STBern} this fundamental connection between the STIT and the Poisson hyperplane tessellations has been established for $j=0$ by martingale techniques and the theory of piecewise deterministic Markov processes. For our purposes we need a slight generalization of this identity for arbitrary $j\in\{0,\ldots,k\}$, the proof of which resembles the argument of Lemma 4 in \cite{TWN} -- designed for $d=3$, $k=1$ and $j=1$.

\begin{lemma}\label{lem:Vjbirthtime}
Let $d\geq 1$, $k\in\{0,\ldots,d-1\}$, $j\in\{0,\ldots,k\}$ and $t>0$. It holds that
$$\bE f(\MP_{k,j}^{(t)})=\int_0^t{(d-j)s^{d-j-1}\over t^{d-j}}\,\bE f(\F_{k,j}^{\PHT(s)})\,\dint s$$
for any non-negative measurable function $f:\cP_k^o\rightarrow\RR$.
\end{lemma}

It is worth pointing out that the density $s\mapsto (d-j)s^{d-j-1}t^{-(d-j)}$ is the marginal density of the last birth time $\beta_{d-k}$ of the $V_j$-weighted typical maximal polytope of dimension $k$ in Theorem \ref{thm:GeneralBirthTime}.

\section{Proof of Theorem \ref{thm:GeneralBirthTime}}\label{sec:proofs}

Throughout this section we let $\L$ be a fixed hyperplane measure with representation \eqref{eq:decomplambda} and $Y(t)$ be a STIT tessellation constructed with the hyperplane measure $\L$ until time $t>0$. The plan of the proof of Theorem \ref{thm:GeneralBirthTime} is first to show the formula for $j=k$ and then to establish the formula in full generality from the special case.

\subsection{The case $j=k$}

\begin{lemma}\label{lem:BirthTimeDensityVk}
Let $d\geq 1$ and $k\in\{0,\ldots,d-1\}$. The joint distribution of the birth-times $\beta_1,\ldots,\beta_{d-k}$ of the volume-weighted typical $k$-dimensional maximal polytope $\MP_{k,k}^{(t)}$ of $Y(t)$ is the uniform distribution on $\{0<s_1<\ldots<s_{d-k}<t\}$, which has density $$(s_1,\ldots,s_{d-k})\mapsto{(d-k)!\over t^{d-k}}\,{\bf 1}\{0<s_1<\ldots<s_{d-k}<t\}.$$
\end{lemma}
\begin{proof}
Recall that $\MP_{k,j}^{(t)}$ is the intersection of $d-k$ maximal polytopes of dimension $d-1$, lying on hyperplanes $H_1,\ldots,H_{d-k}$ and having birth-times $\beta_1,\ldots,\beta_{d-k}$. We are now going to calculate the probability
\begin{equation}\label{eq:proofBirthTimes1}
 \bP\big(\beta_1\in(0,s_1),\ldots,\beta_{{d-k}}\in(s_{d-k-1},s_{d-k})\big),
\end{equation}
where $s_1<\ldots<s_{d-k}$ are fixed. In other words, we want to calculate the probability that hyperplane $H_1$ is born during the time interval $(0,s_1)$, hyperplane $H_2$ is born during the time interval $(s_1,s_2)$ etc.\ until hyperplane $H_{d-k}$ and the time interval $(s_{d-k-1},s_{d-k})$.

To evaluate this probability, we formally mark every member of $\cMP_k^{(t)}$ with its associated $d-k$ birth-times $\beta_1,\ldots,\beta_{d-k}$. This yields a marked point process on the product space $\cP_{k}\times(0,\infty)^{d-k}$, where $\cP_{k}$ is the space of $k$-dimensional polytopes in $\RR^d$. In this context, the probability in \eqref{eq:proofBirthTimes1} is just the mark distribution of this point process evaluated at $(0,s_1)\times\ldots\times(s_{d-k-1},s_{d-k})$. According to the general theory of marked point processes (cf.~\cite[Chapter 3]{SW}), this equals $\r_{k,k}^{(s_1,\ldots,s_{d-k},t)}/\r_{k,k}^{(t)}$, where $\r_{k,k}^{(s_1,\ldots,s_{d-k},t)}$ is the $k$-volume density of those $k$-dimensional maximal polytopes of $Y(t)$ whose birth-times satisfy the constraints $\beta_1\in(0,s_1),\ldots,\beta_{{d-k}}\in(s_{d-k-1},s_{d-k})$. Since $\r_{k,k}^{(t)}=t^{d-k}\r_{k,k}^{(1)}$ according to Lemma \ref{lem:Scaling} a), it remains to determine $\r_{k,k}^{(s_1,\ldots,s_{d-k},t)}$. This is done recursively using that
\begin{equation*}
Y(t)\overset{d}{=}Y(s_1)\boxplus Y(s_2-s_1)\boxplus\ldots\boxplus Y(s_{d-k}-s_{d-k-1})\boxplus Y(t-s_{d-k}),
\end{equation*}
which is valid by iteration stability of $Y(t)$; recall \eqref{eq:iterationstability}. Thus, such a $k$-dimensional maximal polytope must be contained in the intersection of $d-k$ processes of $(d-1)$-dimensional polytopes having $(d-1)$-volume densities $s_1,s_2-s_1,\ldots,s_{d-k}-s_{d-k-1}$, respectively. Thus, by iterated application of intersection formulae for such processes (see \cite[Theorem 4.5.3]{SW}) we find that $\r_{k,k}^{(s_1,\ldots,s_{d-k},t)}$ is proportional to the product $s_1(s_2-s_1)(s_3-s_2)\ldots(s_{d-k}-s_{d-k-1})$ with some proportionality constant $c_1$ only depending on the dimension parameters $k$ and $d$ as well as on the hyperplane measure $\L$. Putting now $c_2:=c_1/\r_{k,k}^{(1)}$, we obtain
$$\bP\big(\beta_1\in(0,s_1),\ldots,\beta_{{d-k}}\in(s_{d-k-1},s_{d-k})\big)={\r_{k,k}^{(s_1,\ldots,s_{d-k},t)}\over\r_{k,k}^{(t)}}=c_2\,{\prod\limits_{j=1}^{d-k}(s_j-s_{j-1})\over t^{d-k}}$$ with the convention that $s_0=0$. Differentiation implies that the joint density of the birth-times is $c_2\,t^{-(d-k)}$. Since this integrates to $1$, we must have $c_2=(d-k)!$, which completes the proof.
\end{proof}

To give a proof for general $j\in\{0,\ldots,k\}$ we need the following fact.

\begin{lemma}\label{lem:independenceKK}
Let $d\geq 1$, $k\in\{0,\ldots,d-1\}$, $f:\cP_k^o\rightarrow\RR$ be non-negative and measurable, $t>0$ and $0<s_1<\ldots<s_{d-k}<t$. Then
\begin{equation}\label{eq:indKK}
\bE\big[f(\MP_{k,k}^{(t)})|\beta_1=s_1,\ldots,\beta_{d-k}=s_{d-k}\big]=\bE\big[f(\F_{k,k}^{\PHT(s_{d-k})})\big].
\end{equation}
\end{lemma}
\begin{proof}
Since $(Y(t))_{t>0}$ -- regarded as a random process taking values in the measurable space of tessellations of $\RR^d$ -- has the Markov property in the continuous time parameter $t$, it holds that
$$\bE\big[f(\MP_{k,k}^{(t)})|\beta_1=s_1,\ldots,\beta_{d-k}=s_{d-k}\big]= \bE\big[f(\MP_{k,k}^{(t)})|\beta_{d-k}=s_{d-k}\big].$$ Next, Lemma \ref{lem:Vjbirthtime} tells us that the conditional distribution of $\MP_{k,k}^{(t)}$, given its last birth time $\beta_{d-k}=s_{d-k}$, is the same as the distribution of $\F_{k,k}^{\PHT(s_{d-k})}$. Formally, this is exactly \eqref{eq:indKK}, which completes the argument.
\end{proof}

\subsection{Proof for general $j$}

The relationship between $\MP_{k,j}^{(t)}$ and $\MP_{k,0}^{(t)}$ can be described by
\begin{equation}\label{eq:ProofGeneralBirthTime1}
\bE f(\MP_{k,j}^{(t)}) = \big[\bE V_j(\MP_{k,0}^{(t)})\big]^{-1}\,\bE\big[f(\MP_{k,0}^{(t)})\,V_j(\MP_{k,0}^{(t)})],
\end{equation}
where $f$ is a non-negative measurable function on $\cP_k^o$. This is a direct consequence of Neveu's exchange formula \cite[Theorem 3.4.5]{SW} and can be shown similarly to Equations (8) and (10) in \cite{Schneider09}. In particular for $j=k$ we have that
\begin{equation}\label{eq:ProofGeneralBirthTime1xx}
\bE f(\MP_{k,k}^{(t)}) = \big[\bE V_k(\MP_{k,0}^{(t)})\big]^{-1}\,\bE\big[f(\MP_{k,0}^{(t)})\,V_k(\MP_{k,0}^{(t)})];
\end{equation}
see also \cite{Calka}. Using \eqref{eq:ProofGeneralBirthTime1xx} with $f(\cdot)\,V_j(\cdot)\,V_k(\cdot)^{-1}$ instead of $f$ there, we find that
\begin{equation}\label{eq:ProofGeneralBirthTime2}
\bE\big[f(\MP_{k,k}^{(t)})\,V_j\,(\MP_{k,k}^{(t)})V_k(\MP_{k,k}^{(t)})^{-1}\big] = \big[\bE V_k(\MP_{k,0}^{(t)})\big]^{-1}\,\bE\big[ f(\MP_{k,0}^{(t)})\,V_j(\MP_{k,0}^{(t)})\big].
\end{equation}
Combining now \eqref{eq:ProofGeneralBirthTime1} with \eqref{eq:ProofGeneralBirthTime2}, we see that $\MP_{k,j}^{(t)}$ and $\MP_{k,k}^{(t)}$ are related by
\begin{equation}\label{eq:ProofGeneralBirthTime3}
\bE f(\MP_{k,j}^{(t)}) = {\bE V_k(\MP_{k,0}^{(t)})\over \bE V_j(\MP_{k,0}^{(t)})}\,\bE\big[f(\MP_{k,k}^{(t)})\,V_j(\MP_{k,k}^{(t)})\,V_k(\MP_{k,k}^{(t)})^{-1}\big].
\end{equation}
Let us fix $0<r_1<\ldots<r_{d-k}<t$ and apply \eqref{eq:ProofGeneralBirthTime3} with $$f=\textbf{1} \big\{\beta_1\in(0,r_1),\ldots,\beta_{d-k}\in(r_{d-k-1},r_{d-k})\big\}$$ to obtain
\begin{equation}\label{eq:ProofGeneralBirthTime4}
\begin{split}
&\bP\big(\beta_1(\MP_{k,j}^{(t)})\in(0,r_1),\ldots,\beta_{d-k}(\MP_{k,j}^{(t)})\in(r_{d-k-1},r_{d-k})\big)\\
&={\bE V_k(\MP_{k,0}^{(t)})\over \bE V_j(\MP_{k,0}^{(t)})}\,\bE\big[{\bf 1}\{\beta_1(\MP_{k,k}^{(t)})\in(0,r_1),\ldots\\
&\hspace{1cm}\ldots,\beta_{d-k}(\MP_{k,k}^{(t)})\in(r_{d-k-1},r_{d-k})\}\,V_j(\MP_{k,k}^{(t)})\,V_k(\MP_{k,k}^{(t)})^{-1}\big].
\end{split}
\end{equation}
Conditioning on the birth-times and using Lemma \ref{lem:BirthTimeDensityVk} yield $$\bE\big[{\bf 1}\{\beta_1(\MP_{k,k}^{(t)})\in(0,r_1),\ldots,\beta_{d-k}(\MP_{k,k}^{(t)})\in(r_{d-k-1},r_{d-k})\}\,V_j(\MP_{k,k}^{(t)})\,V_k(\MP_{k,k}^{(t)})^{-1}\big]$$ $$=\int_0^{r_1}\int_{r_1}^{r_2}\ldots\int_{r_{d-k-1}}^{r_{d-k}}\bE\big[V_j(\MP_{k,k}^{(t)})\,V_k(\MP_{k,k}^{(t)})^{-1}|\beta_1=s_1,\ldots,\beta_{d-k}=s_{d-k}\big]$$ $$\hspace{5cm}\times\,{(d-k)!\over t^{d-k}}\, \dint s_{d-k}\ldots\dint s_2\dint s_1.$$
Lemma \ref{lem:independenceKK} implies that the joint conditional distribution of $(V_k,V_j)$ of the $V_k$-weighted typical $k$-dimensional maximal polytope $\MP_{k,k}^{(t)}$ of $Y(t)$, given its birth-times $(\beta_1,\ldots,\beta_{d-k})=(s_1,\ldots,s_{d-k})$, only depends on the last birth-time $\beta_{d-k}=s_{d-k}$ and equals the joint distribution of $(V_k,V_j)$ of the $V_k$-weighted typical $k$-face $\F_{k,k}^{\PHT(s_{d-k})}$ in a Poisson hyperplane tessellation with intensity measure $s_{d-k}\L$. Whence, due to the scaling property \eqref{eq:PHTScaling} and the homogeneity of the intrinsic volume $V_j$ we infer that
\begin{equation*}
\begin{split}
&\bE\big[V_j(\MP_{k,k}^{(t)})\,V_k(\MP_{k,k}^{(t)})^{-1}|\beta_1=s_1,\ldots,\beta_{d-k}=s_{d-k}\big]\\
&= \bE\big[V_j(\F_{k,k}^{\PHT(s_{d-k})})\,V_k(\F_{k,k}^{\PHT(s_{d-k})})^{-1}\big]\\
&= \bE\big[V_j(s_{d-k}^{-1}\F_{k,k}^{\PHT(1)})\,V_k(s_{d-k}^{-1}\F_{k,k}^{\PHT(1)})^{-1}\big]\\
&= s_{d-k}^{k-j}\;\bE\big[V_j(\F_{k,k}^{\PHT(1)})\,V_k(\F_{k,k}^{\PHT(1)})^{-1}\big]\\
&= c_3s_{d-k}^{k-j},
\end{split}
\end{equation*}
where $c_3:=\bE\big[V_j(\F_{k,k}^{\PHT(1)})\,V_k(\F_{k,k}^{\PHT(1)})^{-1}\big]$ is a constant only depending on $j,k,d$ and $\L$. So,
$$\bE\big[{\bf 1}\{\beta_1(\MP_{k,k}^{(t)})\in(0,r_1),\ldots,\beta_{d-k}(\MP_{k,k}^{(t)})\in(r_{d-k-1},r_{d-k})\}V_j(\MP_{k,k}^{(t)})\,V_k(\MP_{k,k}^{(t)})^{-1}\big]$$
\begin{equation}\label{eq:ProofGeneralBirthTime5}
=\int_0^{r_1}\int_{r_1}^{r_2}\ldots\int_{r_{d-k-1}}^{r_{d-k}}c_3s_{d-k}^{k-j}\,{(d-k)!\over t^{d-k}}\, \dint s_{d-k}\ldots\dint s_2\dint s_1.
\end{equation}
Moreover, according to Lemma \ref{lem:Scaling} b), there is another constant $c_4$, only depending on the dimension parameters $j,k,d$ and on the hyperplane measure $\L$, such that
\begin{equation}\label{eq:ProofGeneralBirthTime6}
{\bE V_k(\MP_{k,0}^{(t)})\over \bE V_j(\MP_{k,0}^{(t)})}=c_4\,t^{j-k}.
\end{equation}
Putting $c_5:=c_3c_4$ and combining \eqref{eq:ProofGeneralBirthTime4} with \eqref{eq:ProofGeneralBirthTime5} and \eqref{eq:ProofGeneralBirthTime6}, we arrive at
\begin{eqnarray*}
& &\bP\big(\beta_1(\MP_{k,j}^{(t)})\in(0,r_1),\ldots,\beta_{d-k}(\MP_{k,j}^{(t)})\in(r_{d-k-1},r_{d-k})\big)\\
&& = \int_0^{r_1}\int_{r_1}^{r_2}\ldots\int_{r_{d-k-1}}^{r_{d-k}}c_5\,(d-k)\,!{s_{d-k}^{k-j}\over t^{d-j}}\, \dint s_{d-k}\ldots\dint s_2\dint s_1.
\end{eqnarray*}
This shows at first that the birth-times of $\MP_{k,j}^{(t)}$ have a joint density, which is given by $$(s_1,\ldots,s_{d-k})\mapsto c_5\,(d-k)!\,{s_{d-k}^{k-j}\over t^{d-j}}{\bf 1}\{0<s_1<\ldots<s_{d-k}<t\}.$$ Since this must integrate to one, we must have $c_5=(d-j)/(d-k)$. This finally proves that $$(s_1,\ldots,s_{d-k})\mapsto (d-j)(d-k-1)!\,{s_{d-k}^{k-j}\over t^{d-j}}{\bf 1}\{0<s_1<\ldots<s_{d-k}<t\}$$ is the joint birth-time density of $\MP_{k,j}^{(t)}$.\hfill $\Box$

\section{Proof of Theorem \ref{thm:GeneralInternalVertices}}\label{sec:proofs2}

We show the formula only for the case of the length-weighted typical maximal segment, the argument for typical maximal segment is similar. The proof follows the idea of the corresponding special cases $d=2$ and $d=3$ in \cite{MNW11,T10,TWN}. The key will be Equation \eqref{eq:keydensity} below, which is simple in the planar case and which has been established in \cite{TWN} for $d=3$ in a long and intricate proof using other methods (parts of which seem to be restricted to small space dimensions). Here, we will employ the power of the distributional identity in Lemma \ref{lem:Vjbirthtime} combined with the result of Theorem \ref{thm:GeneralBirthTime}.

We recall that the length-weighted typical maximal segment of $Y(t)$ is the intersection of $d-1$ maximal polytopes of dimension $d-1$, whose birth-times are denoted by $\beta_1,\ldots,\beta_{d-1}$. Furthermore, let $\ell$ and $\ph$ be the length and the direction of the segment, respectively, where $\ph\in{\cal S}_+^{d-1}$ is the unit vector parallel to the segment. The distribution of $\ph$ is given by a probability measure $\bR$ on ${\cal S}_+^{d-1}$; cf.~\cite[Theorem 1]{HugSchneider11} or \cite[Theorem 4.4.8]{SW}. However, the exact form of $\bR$ is irrelevant for our purposes.

First, by conditioning, we find that
$$\mathsf{p}_{1,1}(n)=\int_0^t\int_0^{s_{d-1}}\ldots\int_0^{s_2}\int_{\cS_+^{d-1}}\int_0^\infty\bP\big(N=n|\ell=x,\ph=u,\beta_1=s_1,\ldots,\beta_{d-1}=s_{d-1}\big)$$ $$\times p_{\ell,\beta_1,\ldots,\beta_{d-1}|\ph=u}(x,s_1,\ldots,s_{d-1})\,\dint x\,\bR(\dint u)\,\dint s_1\ldots\dint s_{d-1},$$
where $N$ is the number of internal vertices and $p_{\ell,\beta_1,\ldots,\beta_{d-1}|\ph=u}$ is the conditional joint density of length $\ell$ and birth-times $\beta_1,\ldots,\beta_{d-1}$, given the direction $u$ of the segment. Two comments are in order. At first, we notice that the joint density $p_{\ell,\ph,\beta_1,\ldots,\beta_{d-k}}$ of length, direction and the $d-k$ birth-times does not need to exist since $\bR$ is not assumed to be absolutely continuous with respect to some reference measure on ${\cal S}_+^{d-1}$. This causes the extra integration with respect to $\bR$ and forces us to condition on $\ph$. It is also worth pointing out that the order of integration in the above expression cannot be changed because of the dependencies among the involved random variables.

We now determine the exact expression for $p_{\ell,\beta_1,\ldots,\beta_{d-1}|\ph=u}$. To start with, write
\begin{equation}\label{eq:proofTheoremVertices1}
\begin{split}
& p_{\ell,\beta_1,\ldots,\beta_{d-1}|\ph=u}(x,s_1,\ldots,s_{d-1})\\
&=p_{\ell|\ph=u,\beta_1=s_1,\ldots,\beta_{d-1}=s_{d-1}}(x)\,p_{\beta_1,\ldots,\beta_{d-1}|\ph=u}(s_1,\ldots,s_{d-1}),
\end{split}
\end{equation}
where $p_{\ell|\ph=u,\beta_1=s_1,\ldots,\beta_{d-1}=s_{d-1}}$ is the conditional density of $\ell$ given that $\ph=u,\beta_1=s_1,\ldots,\beta_{d-1}=s_{d-1}$ and where $p_{\beta_1,\ldots,\beta_{d-1}|\ph=u}$ is the conditional joint density of the birth times $\beta_1,\ldots,\beta_{d-1}$ given the direction $\ph=u$ of the segment. Now, Lemma \ref{lem:independenceKK} applied with $k=1$ shows that $$p_{\ell|\ph=u,\beta_1=s_1,\ldots,\beta_{d-1}=s_{d-1}}(x)=p_{\ell|\ph=u,\beta_{d-1}=s_{d-1}}(x)$$ and that this is the same as the length density of the length-weighted typical edge of the Poisson hyperplane tessellation $\PHT(s_{d-1})$. It is well known (see Example 1.4 in \cite{BaumstarkLast09}) that this length is gamma (Erlang) distributed with parameter $(2,\Lambda(\lk u\rk)s_{d-1})$, whence
\begin{equation}\label{eq:proofTheoremVertices2}
p_{\ell|\ph=u,\beta_1=s_1,\ldots,\beta_{d-1}=s_{d-1}}(x)=\Lambda(\lk u\rk)^2s_{d-1}^2xe^{-\Lambda(\lk u\rk)s_{d-1}x}.
\end{equation}
Moreover, Lemma \ref{lem:independenceKK} also implies that the birth-times are independent of the direction of the segment, i.e.,
\begin{equation}\label{eq:proofTheoremVertices3}
p_{\beta_1,\ldots,\beta_{d-1}|\ph=u}(s_1,\ldots,s_{d-1})=p_{\beta_1,\ldots,\beta_{d-1}}(s_1,\ldots,s_{d-1})={(d-1)!\over t^{d-1}}
\end{equation}
(in fact the direction of $\F_{1,1}^{\PHT(s)}$ does not depend on the intensity $s$; see \cite{HugSchneider11}). Inserting \eqref{eq:proofTheoremVertices2} and \eqref{eq:proofTheoremVertices3} into \eqref{eq:proofTheoremVertices1}, we arrive at
\begin{equation}\label{eq:keydensity}
p_{\ell,\beta_1,\ldots,\beta_{d-1}|\ph=u}(x,s_1,\ldots,s_{d-1})={(d-1)!\over t^{d-1}}\Lambda(\lk u\rk)^2s_{d-1}^2xe^{-\Lambda(\lk u\rk)s_{d-1}x}.
\end{equation}
It remains to determine the conditional probability $$\bP\big(N=n|\ell=x,\ph=u,\beta_1=s_1,\ldots,\beta_{d-1}=s_{d-1}\big).$$ For this, we first notice that for space dimensions $d\ge 3$, the length-weighted typical maximal segment can already have internal vertices at the time of its birth. Suppose now that the length $\ell=x$, the direction $\ph=u$ and corresponding birth-times $\beta_1=s_1,\ldots,\beta_{d-1}=s_{d-1}$ with $0<s_1<\ldots<s_{d-1}<t$ are given. We now gradually reconstruct the internal structure of the segment and exploit the intersection property of STIT tessellations. From time $s_1$ until time $s_2$, behind the first $(d-1)$-dimensional maximal polytope, $\M_1$ say, a number of internal vertices appear, which are Poisson distributed with parameter $\Lambda(\lk u\rk)x(s_2-s_1)$ since the region behind $\M_1$ can be further subdivided during the construction. Next, from time $s_2$ until time $s_3$, behind the first $(d-1)$-dimensional maximal polytope $\M_1$ there is another such polytope $\M_2$ so that internal vertices can be created by further subdivision within $\M_1^+\cap\M_2^+$ and $\M_1^+\cap\M_2^-$, where $\M_1^+$ stands for the half-space behind $\M_1$ (i.e., the last $(d-1)$-dimensional maximal polytope $\M_{d-1}$ appears in $\M_1^-$). Thus, we obtain in the time interval $(s_2,s_3)$ a Poisson distributed number of internal vertices with parameter $2\Lambda(\lk u\rk)x(s_3-s_2)$. In general, from time $s_k$ until time $s_{k+1}$, $1\le k\le d-2$, behind $\M_1$ we have $k-1$ further $(d-1)$-dimensional maximal polytopes $\M_2,\ldots,\M_k$ diving $\M_1^+$ into $k$ regions so that by further subdivision a Poisson distributed number of internal vertices appears, whose parameter is $k\Lambda(\lk u\rk)x(s_{k+1}-s_k)$. During the last time interval $(s_{d-1},t)$, we have not only to consider the number of internal vertices appearing within the $d-2$ regions behind $\M_1$, but also the number of internal vertices which arises in the two cells in
 $\M_1^-$ on the left and on the right of the segment (i.e., in the two regions separated by $\M_{d-1}$). This leads to a Poisson distribution with parameter $d\cdot\Lambda(\lk u\rk)x(t-s_{d-1})$. Adding all the above independent Poisson distributed numbers, we obtain a Poisson distributed number of internal vertices with parameter $\Lambda(\lk u\rk)x(d\cdot t-2s_{d-1}-s_{d-2}-\ldots-s_1)$. Formally, this means that
\begin{equation*}
\begin{split}
&\bP(N=n|\ell=x,\ph=u,\beta_1=s_1,\ldots,\beta_{d-1}=s_{d-1})\\
&={{\big(\Lambda(\lk u\rk)x(d\cdot t-2s_{d-1}-s_{d-2}-\ldots-s_1)\big)^n}\over {n!}}e^{-\Lambda(\lk u\rk)x(d\cdot t-2s_{d-1}-s_{d-2}-\ldots-s_1)}.
\end{split}
\end{equation*}
Thus,
\begin{equation*}
\begin{split}
\mathsf{p}_{1,1}(n) = \int_0^t\int_0^{s_{d-1}}\ldots &\int_0^{s_2}\int_{\cS_+^{d-1}}\int_0^\infty{(d-1)!\over t^{d-1}}\Lambda(\lk u\rk)^2s_{d-1}^2xe^{-\Lambda(\lk u\rk)s_{d-1}x}\\
&\times {{\big(\Lambda(\lk u\rk)x(d\cdot t-2s_{d-1}-s_{d-2}-\ldots-s_1)\big)^n}\over {n!}}\\
&\times e^{-\Lambda(\lk u\rk)x(d\cdot t-2s_{d-1}-s_{d-2}-\ldots-s_1)}\,\dint x\,\bR(\dint u)\,\dint s_1\ldots\dint s_{d-1}.
\end{split}
\end{equation*}
Integrating now with respect to $x$, we see that all terms $\L(\lk u\rk)$ cancel out and we arrive at the desired formula for $\mathsf{p}_{1,1}(n)$.

The value of $\mathsf{p}_{1,1}(n)$ is necessarily independent of $t$ since the number of internal vertices does not change when the tessellation is rescaled. The latter, however, is because of \eqref{eq:STITglobalScaling} the same as a time change, so that $\mathsf{p}_{1,1}(n)$ must be independent of $t$.

The formula for the mean number of internal vertices of the length-weighted typical maximal segment can be obtained from $\displaystyle\sum_{n=0}^\infty n\,\mathsf{p}_{1,1}(n)$ by a straight forward integration procedure.\hfill $\Box$

\end{document}